\documentclass{au} 
\usepackage{graphicx} 
\usepackage{newlattice}

\newcommand{\rgr}{\mathbin{\gr}}
\newcommand{\PSc}{PS${}^\textup{c}$\xspace}
\newcommand{\PScslanted}{PS${}\,^\textup{c}$\xspace} 
\newcommand{\Conj}{\Con_{\textup{J}}}
\newcommand{\Le}{L^\textup{e}}
\newcommand{\conL}[1]{\textup{con}_{L}(#1)}
\newcommand{\conK}[1]{\textup{con}_K(#1)}

\newcommand{\gge}{\gg^\tup{e}}
\DeclareMathOperator{\Prime}{Prime}
\newtheorem{lemma}{Lemma}

\theoremstyle{definition}
\newtheorem{definition}{Definition}

\thanks{This research was supported by
NFSR of Hungary (OTKA), grant number K 115518}

\begin{document}
\title[The General Swing Lemma]{Congruence structure of planar semimodular lattices: \\The General Swing Lemma}

\author[G.\ Cz\'edli]{G\'abor Cz\'edli}
\email[G.\ Cz\'edli]{czedli@math.u-szeged.hu}
\urladdr{http://www.math.u-szeged.hu/~czedli/}
\address{University of Szeged\\ Bolyai Institute\\Szeged,
Aradi v\'ertan\'uk tere 1\\ Hungary 6720}
\author[G.\ Gr\"atzer]{George Gr\"atzer}
\email[G.\ Gr\"atzer]{gratzer@me.com} 
\urladdr[G.\ Gr\"atzer]{http://server.math.umanitoba.ca/homepages/gratzer/}
\author[H. Lakser]{Harry Lakser}
\email[H. Lakser]{hlakser@gmail.com} 
\address{Department of Mathematics\\University of Manitoba\\Winnipeg, MB R3T 2N2\\Canada}
\dedicatory{To the memory of E.\,T. Schmidt}

\begin{abstract} 
The Swing Lemma (G.\ Gr\"atzer, 2015) 
describes how a congruence spreads 
from a prime interval to another 
in a slim (having no $\SM 3$ sublattice), 
planar, semimodular lattice.
We generalize the Swing Lemma to planar semimodular lattices.
\end{abstract}

\leftline{\hfill  \emph{Version of March 3, 2017}}

\maketitle

\section{Introduction}\label{S:Introduction}
Congruences in a distributive lattice $L$ 
were investigated in G. Gr\"atzer and E.\,T. Schmidt~\cite{GS58d}:
for $a \leq b, c \leq d \in L$,
the interval $[a,b]$ congruence spreads to the interval $[c,d]$
(that is, $\cng c = d (\con{a,b})$ or equivalently, $\con{a,b} \geq \con{c,d}$) 
if{}f we can get from $[a, b]$ to $[c, d]$
with one up-step (joining with an element) and one down-step (meeting with an element).
See G.~Gr\"atzer~\cite{gG66} for an interesting application of this idea.
 
In G. Gr\"atzer and E.\,T. Schmidt~\cite{GS96a} and
G. Gr\"atzer, H. Lakser, and E.\,T. Schmidt \cite{GLS98a},
we succeeded in representing a finite distributive lattice $D$
as the congruence lattice of a finite semimodular lattice $L$.
It was a great surprise that the lattice $L$ constructed was planar.

This result started the study of \emph{planar semimodular} lattices
in G.~Gr\"atzer and E. Knapp~\cite{GKn09}, \cite{GKn10};
see also G.~Cz\'edli and E.\,T.~Schmidt~\cite{CS12}, 
\cite{CS13}, G.~Cz\'edli~\cite{cG12}. 
For a 2013 review of this field, see G. Cz\'edli and G. Gr\"atzer \cite{CGa}.
Part VII in the book 
G.~Gr\"atzer~\cite{CFL2} provides a 2015 review of this field.
More articles on these topics are listed in the Bibliography.

It is a special focus of this new field to examine how
a prime interval $\fp$ congruence spreads to another prime interval $\fq$.

The first main result of this type, 
the Trajectory Coloring Theorem for Slim Rectangular Lattices,
is due to G. Cz\'edli~\cite{gC14}. It gives a quasiordering
of trajectories of slim rectangular lattices (a subclass of slim, planar, 
semimodular lattices) $L$ to represent the ordered set
of join-irreducible congruences of $L$.

The second main result of this type is the Swing Lemma for slim
(having no~$\SM 3$ sublattice), planar, 
semimodular lattices in G. Gr\"atzer~\cite{gG15}, see Section~\ref{S:General}. 

In this paper, we extend the Swing Lemma to planar semimodular lattices.

We use two different approaches.

The first, utilizes the first author's Eye Lemma \cite[Lemma 5.2]{cG12},
a contribution to the theory of congruences of planar lattices.
This lemma is applied to planar semimodular lattices
in Lemma~\ref{L:GSW1}.

The second approach is a contribution 
to the theory of congruences of finite lattices. 
Instead of the Eye Lemma, we prove the Tab Lemma
for finite lattices, which is based 
on the Prime-Projectivity Lemma of the second author
\cite[Lemma 4]{gG14c}. 
We state and prove Lemma~\ref{L:GSL}, 
a general version of the Swing Lemma
that does not require planarity.

\subsection{Outline}\label{S:Outline}
The basic concepts and results are introduced in 
Section~\ref{S:Preliminaries}.

In Section~\ref{S:General}, 
we state the Swing Lemma and the General Swing Lemma.
Section~\ref{S:short} presents
our short approach through planar lattices, utilizing the Eye Lemma.
The Tab Lemma for finite lattices is stated and proved
for general finite lattices in Section~\ref{S:tab}, 
which is then applied in Section~\ref{S:Swing} 
to prove the General Swing Lemma. 
Finally, Section~\ref{S:Discussion} discusses some related topic,
including a formulation of the General Swing Lemma
that makes no reference to planarity.

\section{Preliminaries}\label{S:Preliminaries}

We use the basic concepts and notation as in the books 
G. Gr\"atzer~\cite{LTF} and~\cite{CFL2}.

We call a planar semimodular lattice a \emph{PS lattice} and
a slim, planar, semimodular lattice an \emph{SPS lattice}. 
Note that a planar lattice is finite  by definition. 

For a lattice $L$, let $\Prime(L)$ denote the set of prime intervals of $L$.

A \emph{multi-diamond} is a lattice $M$ 
isomorphic to the $(n+2)$-element modular lattice~$\SM n$  
of length $2$ for some integer $n \geq 3$. 
A~\emph{covering multi-diamond} $M$ in a lattice $L$ is an \emph{interval} in $L$ isomorphic to some lattice $\SM n$
 for some integer $n \geq 3$.
An~element $m \in L$ is a \emph{tab}
of $L$ if it is doubly-irreducible 
and it is in a covering multi-diamond in $L$.
Note that $m^*$, the unique upper cover of $m$, is the unit
of the multi-diamond and $m_*$, 
the unique lower cover of $m$, is the zero. 

\begin{definition}
Let $\fp$, $\fq$ be prime intervals in a lattice $L$. 
We say that $\fp$ \emph{swings to}~$\fq$ if the following conditions are satisfied:
\begin{enumeratei}
\item $\fp \neq \fq$;
\item $1_{\fp} = 1_{\fq}$;
\item there is a lower cover $a \neq 0_{\fp}, 0_{\fq}$ of $1_{\fp}$ such that the three element set $\set{0_{\fp}, 0_{\fp}, a}$
generates a sublattice $S$ of $L$ isomorphic to
$\SfS 7$, where $0_{\fq}$ is the unique dual atom of $S$ that is
a proper join---see Figure~\ref{s7F}.
\end{enumeratei}
\end{definition}

\begin{figure}[htb]
\centerline{\includegraphics{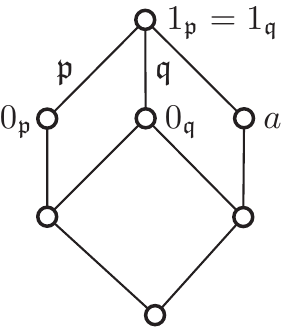}}
\caption{$\mathfrak{p}$ swings to $\mathfrak{q}$}\label{s7F}
\end{figure}

This swing concept---modified from G. Gr\"atzer~\cite{gG15}, 
where it is defined for SPS lattices---makes sense 
in an arbitrary lattice, 
and we get the following statement for general lattices, 
equally as trivial as it is for SPS lattices.

\begin{lemma}\label{collapseL}
Let $L$ be any lattice, let $\fp, \fq$ be prime intervals in $L$, and let $\fp$
swing to $\fq$. Then $\fp$ congruence spreads to $\fq$.
\end{lemma}

We shall need the following result later on.
\begin{lemma}\label{genSL}
Let $L$ be a finite semimodular lattice, let $[o,i]$ be an interval of length $2$ in $L$, and let $m$
be an atom of $[o,i]$ that is join-irreducible in $L$. Let $x \nin [o,i]$ be a lower cover of $i$ in $L$.
Then there is an atom $a \in [o,i]$ such that the elements $x, a, m$ generate a sublattice $S$ of
$L$ that is isomorphic to $\SfS 7$, where $a$ is the dual atom of $S$ that is a proper join.
 \end{lemma}
\begin{proof}
Since $x \ngeq m$, $u = x \mm m < m$. Since $m$ is join-irreducible in $L$, it has exactly one lower cover $o$, hence, $u \leq o$. 
Since $x \ngeq o$, we conclude that $u < o$.

Since $x$ is a lower cover of $i$, it must be distinct from $u$ 
and so $u < x$. Then there is an
upper cover $v \leq x$ of $u$. Since $x \mm o = u$, 
it follows that $v \nleq o$---see Figure~\ref{genSF}.
Thus, by semimodularity, $a=v \jj o \leq i$ covers $o$. 
Since $a$ is proper join, it is distinct from $m$. 
Note that $a \in [o,i]$, so it is distinct from $x$. 
Since the length of $[o,i]$ is $2$, 
it follows that $a$ is a lower cover of $i$. 
Set $y = x \mm a \geq v$.

Then the subset $S = \set{u,y,o,x,a,m,i}$ is the sublattice generated by the set of lower covers
$\set{x,a,m}$ of $i$, and is isomorphic to $\SfS 7$ with the dual atom $a$ a proper join. \

Indeed, as noted above, the lower covers $x$, $a$, $m$ of $i$ are all distinct.

Furthermore, $S$ is a meet subsemilattice of $L$. 
By definition, $x \mm a = y$; 
also $a \mm m = o$, since they are distinct upper covers of $o$, and $x \mm m = u$, again by~definition.

That $S$ is also a join subsemilattice of $L$ is also easy:
\begin{equation*}
x \jj a = a \jj m = x \jj m = i
\end{equation*}
since
$x, a, m$ are distinct lower covers of $i$. 
Note also that by definition, $a=v \jj o$ and also
\[
   a = v \jj o \leq y \jj o \leq a,
\]
that is, that $y \jj o = a$.

The other two incomparable joins are then immediate:
\begin{equation*}
y \jj m = y \jj o \jj m = a \jj m = i,
\end{equation*}
and, similarly,
\begin{equation*}
o \jj x = o \jj y \jj x = a \jj x = i.
\end{equation*}
See Figure~\ref{genSF}, and note that, although $v$ appears in the figure, it is not an element of $S$.
\end{proof}

\begin{figure}[h]
\includegraphics[width=0.50\textwidth]{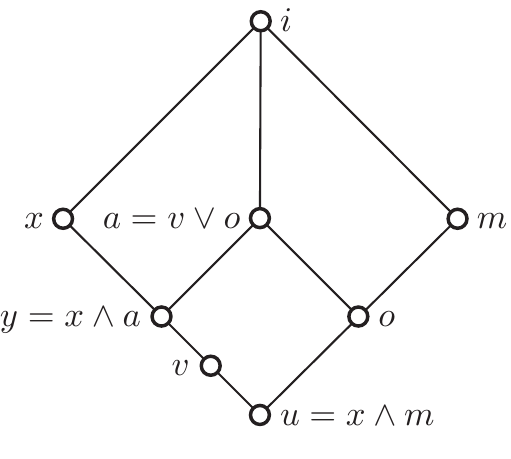}
\caption{Lemma~\ref{genSL}: Generating the $\SfS 7$.}\label{genSF}
\end{figure}

For prime intervals $\fp, \fq$ in an SPS lattice $L$, 
if $\fp$ congruence spreads to $\fq$, 
then the Swing Lemma gives a special kind of projectivity. 
We call such a projectivity an SPS projectivity.
Since the concept  ``$\fp$ swings to $\fq$'' now makes
sense in any lattice, we can define the concept of SPS projectivity for any
lattice.

\begin{definition}\label{SPSD}
Let $L$ be a lattice, and let $\fp, \fq$  be prime intervals in $L$. 
We say that $\fp$ is
SPS \emph{projective to} $\fq$ if either $\fp = \fq$ 
or there exists a prime interval $\fr$  
with $\fp$ up-perspective to  $\fr$, 
and a sequence of prime intervals and a sequence of binary relations
\begin{equation}\label{Eq:SPS}
   \fr = \fr_0 \rgr_1 \fr_1 \rgr_2 \fr_2 \cdots \fr_{n-1} \rgr_n \fr_n = \fq,
\end{equation}
where each relation $\gr_k$ is either a down perspectivity 
of $\fr_k$ to $\fr_{k+1}$ or a swing of $\fr_k$ to $\fr_{k+1}$.
\end{definition}

We call \eqref{Eq:SPS} an \emph{SPS sequence}.
By Lemma~\ref{collapseL}, the following holds in any lattice.

\begin{lemma}\label{collapse2L}
Let $L$ be any lattice, and let $\fp$ and $\fq$ be prime intervals in $L$. If $\fp$ is
SPS projective to $\fq$, then $\fp$ congruence spreads to $\fq$.
\end{lemma}

We will need the following results (G. Gr\"atzer~\cite{gG14b} 
and \cite[Lemma~3.3]{CFL2}):
\begin{lemma}\label{l33L}
Let $L$ be a finite lattice. Let $\bgg$ be an equivalence relation on $L$ all of whose
equivalence classes are intervals. Then $\bgg$ is a congruence relation if{f} the following
condition and its dual hold:
\begin{equation*}
\text{If $x$ is covered by $y \neq z \in L$ and $\cngd x = y (\bgg)$, 
then $\cngd z = y \jj z (\bgg)$.}
\end{equation*}
\end{lemma}

This lemma is easy to prove, but it is useful. 
The next lemma is far deeper
(G.~Gr\"atzer~\cite{gG14c} and \cite[Lemma~24.1]{CFL2}).
To state it, we need the concept of prime-perspectivity.
Recall that for (not necessarily prime) intervals $I = [0_I,1_I]$ and
$J=[0_J,1_J]$ we say that $I$ is \emph{down perspective} to $J$,
in formula $I \perspdn J$, if $0_I \jj 1_J =  1_I$ and $0_I \mm 1_J = 0_J$.

Let $L$ be a finite lattice 
and let $\fp$ and $\fq$ be prime intervals of $L$.  
\begin{figure}[b!]
\centerline{\includegraphics{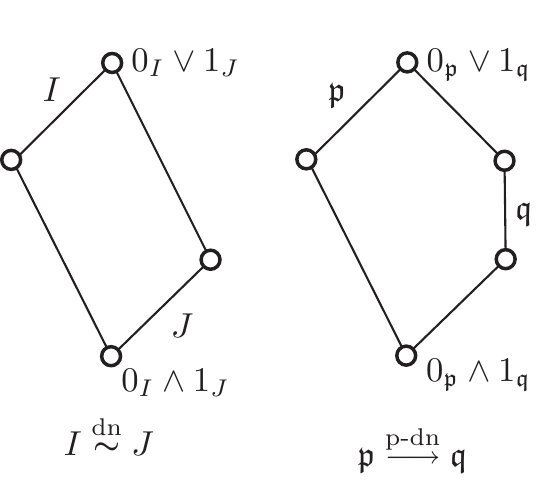}}
\caption{Introducing prime-perspectivity}\label{F:intro}
\end{figure} 
Then the binary relation $\fp$ \emph{prime-perspective down} to $\fq$, 
in formula, $\fp \pperspdn \fq$, is defined as
$\fp \perspdn [0_\fp \mm 1_\fq, 1_\fq]$
(that is, $0_\fp \jj 1_\fq = 1_\fp$)
and $0_\fp \mm 1_\fq \leq 0_\fq$; 
we define \emph{prime-perspective up}, $\fp \pperspup \fq$, 
dually---see Figure~\ref{F:intro}.
Let \emph{prime-perspective}, $\fp \ppersp \fq$, 
mean that $\fp \pperspup \fq$ or $\fp \pperspdn \fq$ and let
\emph{prime-projective}, $\fp \pproj \fq$, 
be the transitive extension of $\ppersp$. 

\begin{lemma}[Prime-Projectivity Lemma]\label{L:prime}
Let $L$ be a finite lattice and 
let $\fp$ and $\fq$ be distinct prime intervals in $L$.  
Then $\fq$ is collapsed by $\con{\fp}$ 
if{}f $\fp \pproj \fq$, that is, 
if{}f there exists a sequence of pairwise distinct prime intervals
$\fp = \fr_0, \fr_1, \dots, \fr_n = \fq$ satisfying
\begin{equation}\label{E:ppthm}
\fp = \fr_0 \ppersp \fr_1 \ppersp \dotsm \ppersp \fr_n = \fq.
\end{equation}
\end{lemma}

\section{The Swing Lemma and the General Swing Lemma}\label{S:General}

Now we state the main result in G. Gr\"atzer~\cite{gG15}. For a second proof, see G.~Cz\'edli~\cite{cGa} and for the shortest proof, G. Cz\'edli and G. Makay~\cite{CM}.

\begin{lemma}[Swing Lemma]\label{swingL}
Let $L$ be a slim, planar, semimodular lattice
and let $\fp, \fq$ be prime intervals in $L$. 
Then~$\fp$ congruence spreads to $\fq$ if{f} $\fp$ is SPS projective to $\fq$.
\end{lemma}

Next we define the concept of \emph{switching}.
\begin{definition}\label{swD}
Let $L$ be any finite lattice, 
and let $\fp$, $\fq$ be distinct prime intervals in $L$. 
We say that $\fp$ \emph{switches to} $\fq$  
if both $\fp$ and $\fq$ lie in a common covering $\SM 3$ in $L$.
\end{definition}

Observe that switching is a symmetric concept, and, furthermore,
we have the immediate result:
\begin{lemma}\label{m3L}
Let $L$ be a lattice. 
If the prime interval $\fp$ switches to the prime interval $\fq$ in $L$,
then $\conL{\fp} = \conL{\fq}$.
\end{lemma}

We now generalize the Swing Lemma to PS lattices.

\begin{lemma}[General Swing Lemma]\label{L:GSL}

Let $L$ be a planar semimodular lattice, and let $\fp, \fq$ be distinct prime
intervals in $L$. Then $\fq$ is collapsed by $\con{\fp}$ if{f} there exist sequences of
prime intervals in $L$,
$\fp_0 = \fp, \fp_1, \dots, \fp_n$ and
$\fq_0, \fq_1, \dots, \fq_n = \fq$
where $\fp_k$ is SPS projective to $\fq_k$
for $k = 0,\dots, n$, and, provided that $n > 0$, $\fq_k$ switches to $\fp_{k+1}$ for
$k=0,\dots, n-1$.
\end{lemma}

Since there is no switch in a slim lattice, the Swing Lemma is the General Swing Lemma for slim lattices.

\section{A short approach through planar lattices}\label{S:short}

Let $L$ be a planar lattice with an associated planar diagram.
\text{A~$4$-cell} $S = \set{0_S,v,w,1_S}$ in $L$ 
is a covering square with no internal element.
We \emph{insert an eye} $e$ into $S$ by adding an element $e$
to $L$, turning the covering square $S$ into an $\SM 3$ 
so that the new element $e$ becomes 
an internal element of $[0_S,1_S]$ in the lattice $L \uu \set{e}$, 
denoted by $\Le$. 

A quasiordered set $(H,\nu)$ is a nonempty set $H$ 
with a binary relation $\nu$ that is reflexive and transitive.
Following G. Cz\'edli~\cite{cG12}, a \emph{quasi-coloring} of  a lattice $M$ of finite length is a surjective map $\gg$  from  $\Prime(M)$ onto a 
quasiordered set $(H,\nu)$ such that for $\fa,\fb\in\Prime(M)$, 
the following two conditions hold:
\begin{enumerate}
\item[(C1)] if $(\gg\fa, \gg\fb)\in\nu$, then  
$\fb$ congruence spreads to $\fa$;
\item[(C2)] if $\fb$ congruence spreads to $\fa$, then $(\gg\fa,\gg\fb)\in\nu$.
\end{enumerate}

For $\rho\ci Q^2$, there is a least quasiorder $\ol \gr$
containing $\rho$ on $Q$;
it is the reflexive and transitive extension of $\rho$ on~$Q$. 
Let $(Q,\nu)=(\Conj L, \ci)$, and let
$\gg\colon \Prime(L)\to (Q,\nu)$ be the  
\emph{natural quasi-coloring} of $L$ 
defined by the rule $\gg\fa=\conL{\fa}$;  
see \cite[above Lemma 2.1]{cG12}.  
Let $\bga=\gg[v,1_S]=\conL {v,1_S}$ and  
$\bgb=\gg[w,1_S]=\conL{w,1_S}$.  
Define $\tau$ as 
the least quasiorder 
containing  $\nu\uu\set{(\bga,\bgb),(\bgb,\bga)}$, 
and extend $\gg$ to $\gge\colon\Prime(\Le)\to (Q,\tau)$ 
by $\gg [0_S, e] = \gg [e, 1_S] = \bga$.
(G. Cz\'edli \cite{cG12} uses the notation 
$L^\circledcirc$ and $\gg^\circledcirc$.)

Now we can state the crucial technical lemma G. Cz\'edli \cite[Lemma 5.2]{cG12}.

\begin{lemma}[Eye Lemma]\label{L:eye}
The map $\gge$ is a quasi-coloring of $L^\tup{e}$.
\end{lemma}
 
Let $L$ be a planar semimodular lattice and let $\fp, \fq$ be distinct prime
intervals in $L$. 
Let us call the sequence provided for $\fp, \fq$ in the General Swing Lemma a \emph{PS sequence}. 
If we assume, in addition, 
that every perspectivity is a perspectivity in a $4$-cell,
we call the sequence a \emph{PS cell sequence} (\PSc sequence, for short).
%
More exactly, in a planar semimodular lattice $L$, 
for the prime intervals $\fp$ and $\fq$ in $L$,
a \PSc sequence from $\fp$ to $\fq$ is a sequence 
$\fr_0=\fp$, $\fr_1$, \dots, $\fr_{n-1}$, $\fr_n=\fq$ 
of prime intervals of $L$ such that, 
for each $1 \leq i \leq n$, one of the following three conditions holds.
\begin{enumeratei}
\item\label{slsa} $\fp_{i-1}$ and $\fp_i$ are \emph{cell perspective},
that is, they are opposite sides of the same $4$-cell;
\item\label{slsb} ${\fp_{i-1}}$ \emph{swings} to ${\fp_{i}}$;
\item\label{slsc} ${\fp_{i-1}}$ \emph{switches} to ${\fp_{i}}$.
\end{enumeratei}

The following statement is well known.

\begin{lemma}\label{L:inM}
Let $L$ be a planar lattice and 
let $S = \set{0_S,v,w,1_S}$ be a~$4$-cell in~$L$. 
If we insert an eye $e$ into $S$, 
then any two prime intervals of $S \uu \set{e}=\SM 3$ 
are connected with 
a \PScslanted{} sequence.
\end{lemma}

We are now ready to prove the General Swing Lemma in the following equivalent form.

\begin{lemma}[General Swing Lemma$'$]\label{L:GSW1}
Let $L$ be a planar semimodular lattice and let $\fp, \fq$ be distinct prime
intervals in $L$. Then $\fp$ congruence spreads to~$\fq$ if{f} 
there exists 
a \PScslanted sequence from $\fp$ to $\fq$.
\end{lemma}

\begin{proof} The sufficiency is straightforward by Lemma~\ref{collapse2L}. 

We prove the necessity by induction on the number of eyes in $L$. 
If this number is $0$, then $L$ is slim, 
so the statement follows from the Swing Lemma. 

So assume the validity of the lemma for a planar semimodular lattice $L$ and 
insert an eye $e$ into a $4$-cell 
$S=\set{0_S,v,w,1_S}$.
Let $M = S \uu \set{e}$.
We want to prove that the General Swing Lemma$'$ holds in $L^\tup{e}$.
To accomplish this assume that $\fp,\fq$ are prime intervals of $L^\tup{e}$ 
and $\fp$ congruence spreads to $\fq$. 
We have to find a \PSc sequence from $\fp$ to $\fq$. 

Let us assume that 
\begin{equation}\label{E:pq}
   \fp, \fq \ci L.
\end{equation}
By the Eye Lemma, $\gge$ is a quasi-coloring. 
Since $\fp$ congruence spreads to $\fq$ in $L^\tup{e}$,  
(C2) yields that $(\gge \fq, \gge\fp) \in \tau$.
Hence, there is a sequence 
$\bgd_0=\gge\fq$, $\bgd_1, \dots, \bgd_{n-1}, \bgd_n=\gge\fp$ 
of elements of~$Q=\Conj L$ such that  
$(\bgd_{i-1},{\bgd_i}) \in 
\set{(\bga,\bgb),(\bgb,\bga)} 
\uu \nu$ for $i\in\set{1,\dots,n}$. 
Let $\fr_0=\fq$ and $\fr_n=\fp$.  
For each $i\in\set{1,\dots,n-1}$, 
the surjectivity of $\gg$ allows us 
to pick a prime interval $\fr_i\in\Prime(L)$ such that $\gg\fr_i=\bgd_i$; 
note that this equality holds also for $i\in\set{0,n}$, 
since $\fp,\fq\in\Prime(L)$ and $\gge$ extends $\gg$. 
For  $i\in\set{1,\dots,n}$, there are two cases.

First, assume that $(\bgd_{i-1},{\bgd_i}) \in \nu$. 
Since $(\gg\fr_{i-1},\gg\fr_{i}) 
= (\bgd_{i-1},{\bgd_i}) \in \nu$,  
(C1) gives that $\conL{\fr_{i-1}} \ci \conL{\fr_{i}}$, 
whereby the induction hypothesis yields a \PSc sequence 
$\vec s_i{}'$ from $\fr_i$ to $\fr_{i-1}$ in $L$. 
If two consecutive members of $\vec s_i{}'$ belong to $\Prime(S)$, 
then $\vec s_i{}'$ is not a \PSc sequence in $L^\tup{e}$, 
but Lemma~\ref{L:inM} helps in turning it into a \PSc sequence 
$\vec s_i$ from $\fr_i$ to $\fr_{i-1}$ in $L^\tup{e}$. 

Second, assume that $(\bgd_{i-1},{\bgd_i}) = (\bga, \bgb)$ or $(\bgb,\bga)$. 
By reason of symmetry, we can assume that $(\bgd_{i-1}, \bgd_i) = (\bga, \bgb)$. 
Since $\gg(\fr_{i-1})=\bgd_{i-1}=\bga=\gg[v,1_S]$, it follows that 
$\conL{\fr_{i-1}} \ci \conL{v,1_S}$ by (C1). 
As in the first case, the induction hypothesis 
and Lemma~\ref{L:inM} yield  a \PSc sequence~$\vec s_i{}'$  
from $[v,1_S]$ to $\fr_{i-1}$ in $L^\tup{e}$. 
Similarly, we have a \PSc sequence $\vec s_i{}''$ 
from $\fr_i$ to $[w,1_S]$ in $L^\tup{e}$. 
Since Lemma~\ref{L:inM} gives a \PSc sequence 
$\vec s_i{}'''$ from $[w,1_S]$ to~$[v,1_S]$ in $L^\tup{e}$, 
the concatenation of $\vec s_i{}''$, $\vec s_i{}'''$, and $\vec s_i{}'$ is 
a \PSc sequence $\vec s_i{}$ from $\fr_i$ to $\fr_{i-1}$ in $L^\tup{e}$.

Finally, the concatenation of $\vec s_n$, $\vec s_{n-1}$, \dots, $\vec s_1$ 
is a \PSc sequence from $\fp$ to $\fq$ in $L^\tup{e}$, as required.

It remains to prove this lemma if \eqref{E:pq} fails.
In this case, we have that $e \in \fp \uu \fq$. 
If $e \in \fp$, then $\fp \in \Prime(M)$. Let $\fp'$ be any prime interval
in $S \ci L$. Then $\con{\fp'} = \con{\fp}$ in $L^\tup{e}$ and,  
by Lemma~\ref{L:inM}, there is a \PSc sequence from $\fp$ to $\fp'$ in
$L^\tup{e}$. If $e \nin \fp$, just set $\fp' = \fp$.
Similarly, if $e \in \fq$, then, for any prime interval $\fq'$ in $S$,
$\con{\fq'} = \con{\fq}$ and there is
a \PSc sequence from $\fq'$ to $\fq$ in~$L^\tup{e}$. Otherwise, set $\fq'=\fq$. Then
$\fp'$ congruence spreads to $\fq'$ in $L^\tup{e}$ and
$\fp'$, $\fq'$ satisfy~\eqref{E:pq}. Thus there is a \PSc sequence
from $\fp'$ to $\fq'$ in $L^\tup{e}$. 
Combining these \PSc sequences, we get a
\PSc sequence in $L^\tup{e}$ from $\fp$ to $\fq$.
\end{proof}

\section{The Tab Lemma for finite lattices}\label{S:tab}

We now state our main technical lemma for finite lattices.

\begin{lemma}[Tab Lemma]\label{mainL}
Let $L$ be a finite lattice, and let $m$ be a tab of $L$
in the covering multi-diamond $[o,i]$. 
Let $K= L - \set{m}$, a~sublattice of $L$. 
Let $\bga \in \Con{K}$ and let $\bgb = \conL{\bga}$. 
Then 
\[
   \bgb\restr K =
      \begin{cases}
         \bga,                  &\text{if $\ncng o = i(\bgb)$};\\
         \bga \jj_K \conK{o,i}, &\text{otherwise.}
   \end{cases}
\]
\end{lemma}

Note that the first case means that $\bgb$ 
is an extension of $\bga$.
\begin{proof}
We first consider the case $\ncng o = i (\bgb)$. 
We extend the congruence relation $\bga$ of $K$
to the equivalence relation $\bgg$ of $L$ 
by making the equivalence class of $m$ the singleton
$\set{m}$.
We observe that we cannot have $x < m < y$ 
for some $x$, $y \in K$ with $\cng x = y (\bga)$;
indeed, $x \leq m_* < m < m^* \leq y$, so, 
since $o= m_*$ and $i = m^*$, we would have
$\cng o = i (\bga)$ and thereby  $\cng o = i (\bgb)$.

Thus all the equivalence classes of $\bgg$ are
intervals in $L$ (the intervals in $K$ and $\set{m}$)
and we can apply Lemma~\ref{l33L} 
to show that $\bgg$ is a congruence relation on $L$. 

Let $x, y, z$ be distinct elements of $L$ 
with $y, z$ covering $x$ and $\cng x = y (\bgg)$.
Since the equivalence class of $m$ is a singleton, it follows that $x \neq m$ and $y \neq m$.
But $z = m$ is also impossible. Indeed, $\cng x = y (\bga)$ and so $\cng x = y (\bgb)$ and $\bgb$
is a congruence relation on $L$, whereby $\cng y \jj z = z (\bgb)$, and $y \jj z > z$. Since $m^* = i$, 
from $z=m$ it would follow that $\cng i = m (\bgb)$, that is, 
$\cng o = i (\bgb)$,
contradicting our hypothesis that  
$\ncng o = i (\bgb)$---see~Figure~\ref{zismF}.

Consequently,
$x$, $y$, $z$ are all elements of $K$, and $\cng x = y (\bga)$. 
Then $\cng z = y \jj z (\bga)$,
that is $\cng z = y \jj z (\bgg)$. This argument and its dual show that $\bgg$ is a congruence relation
on $L$.

Then, it is immediate from its definition that 
$\bgg = \conL{\bga} = \bgb$ and $\bgb\restr K = \bga$.

It remains to consider the case $\cng o = i (\bgb)$. 
Let us denote the congruence relation
$\bga \jj_K \conK{o,i}$ of $K$ by $\bga'$. 
We extend $\bga'$ to  an equivalence relation $\bgg$ on $L$
by setting $\cng x = y (\bgg)$ for all $x,y \in [o,i]$. 
Since $o = m_*$ and $i = m^*$, and $\cng o = i (\bga')$ 
it follows that all
equivalence classes of $\bgg$ are intervals in~$L$. 
We then apply Lemma~\ref{l33L} to $\bgg$
to show that $\bgg$ is a congruence relation on~$L$. 

Let $x, y, z$ be distinct elements of $L$ with $y$, $z$ covering $x$ and $\cng x = y (\bgg)$.

If $x, y, z$ are all elements of $K$, then $\cng x = y (\bga')$ and so $\cng y = y \jj z (\bga')$,
that is $\cng y = y \jj z (\bgg)$.

The case $x = m$ is impossible, since $m$ has only one upper cover.

The case $y = m$ implies that $z \in K$ and that $o = x < z$, since $o = m_*$. 
Since $y \jj z > y =m$, we have that
$y \jj z \geq i > y$. Thus
\begin{equation*}
 \cngd {y \jj z = i \jj z} ={o \jj  z = z} (\bga'),
\end{equation*}  
that is, $\cng y \jj z = z (\bgg)$---see Figure~\ref{yismF}.

The case $z = m$  implies that $y \in K$ and $x = o$. 
Thus $\cng y = o (\bga')$. Furthermore,
$y \jj z \geq i$ since $y \jj z > z = m$, whereby
\begin{equation*}
\cngd {y \jj z = y \jj i} = {o \jj i = i} (\bga'),
\end{equation*}
that is, $\cng y \jj z = i (\bgg)$---see Figure~\ref{zism2F}.
But $\cng i = m (\bgg)$ by the definition of~$\bgg$. 
Consequently, $\cng y \jj z = {m = z}(\bgg)$ since $\bgg$ is an equivalence relation.

Thus by the above argument and its dual, $\bgg$ is a congruence relation on $L$.

Now $\bga \ci \bgg$. Thus $\bgb =\conL{\bga} \leq \bgg$. Since, clearly,
\begin{equation*}
\bgg\restr K = \bga' = \bga \jj_{K} \conK{o,i},
\end{equation*}
 it follows that
$\bgb\restr K = \bga' =  \bga \jj_{K} \conK{o,i}$.
\end{proof}

Now we shall discuss some applications of the Tab Lemma.
In these discussions, $L$ is a finite lattice, 
$m$ is a tab in a covering multi-diamond $[o,i]$ of $L$, 
and the sublattice $K = L - \set{m}$.

We first make a very simple observation.
\begin{lemma}\label{simpL}
Let $\fp$ and $\fq$ be prime intervals in $K$. 
If $\fp$ is prime-perspective down
to $\fq$ in $L$, then $\fp$ is prime-perspective down to $\fq$ in $K$, 
and dually for prime-per\-spective~up.
\end{lemma}

\begin{figure}[t!]
\centerline{\includegraphics{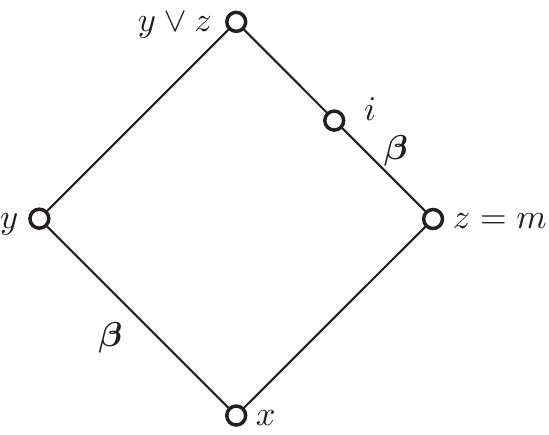}}
\caption{Case $\ncng o = i (\bgb)$ and $z = m$ of Lemma~\ref{mainL}}\label{zismF}

 \bigskip
 
 \bigskip

\includegraphics{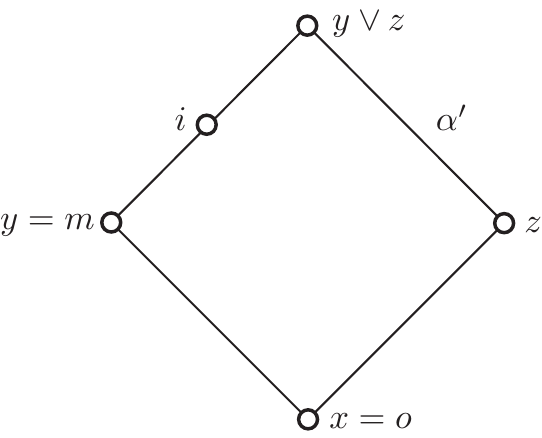}
\caption{Case $\cng o = i (\bgb)$ and $y = m$ of Lemma~\ref{mainL}}\label{yismF}

 \bigskip
 
 \bigskip

\centerline{\includegraphics{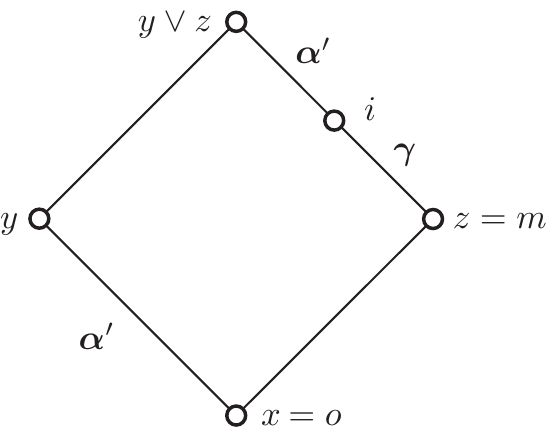}}
\caption{$\cng o = i (\bgb)$ and $z = m$ case of Lemma~\ref{mainL}}\label{zism2F}
\end{figure}

\begin{proof}
Clearly, we may assume that $\fp$ and $\fq$ are distinct intervals. 
Now the end-points $0_{\fp}$,
$1_{\fp}$, $0_{\fq}$, $1_{\fq}$ are all distinct from $m$. 
There is at most one other
element involved, which is $0_{\fp} \mm 1_{\fq}$, 
which is meet-reducible,
and so is also distinct from $m$. 
Thus the prime-perspectivity occurs in $K$.
\end{proof}

\begin{lemma}\label{sm0L}
Let $\fp$ be a prime interval in $K$. Then $\cng o = i (\conL{\fp})$ if{f} there is a
prime interval $\fq$ in $K$ with $\cng 0_{\fq} = 1_{\fq} (\conK{\fp})$ and either
$1_{\fq} = i$ or, dually, $0_{\fq} = o$.
\end{lemma}
\begin{proof}
First, assume that there is such a prime interval $\fq$ with $1_{\fq} = i$. Now 
$0_{\fq} \neq m$ and so $0_{\fq} \jj m = 1_{\fq}$. Furthermore,
$0_{\fq} \mm m < m$. Since $m_* = o$ in~$L$, we have
$0_{\fq} \mm m \leq o$, that is, the prime interval $\fq$ is 
prime-perspective down to
the prime interval $[o,m]$. Since $\cng 0_{\fq} = 1_{\fq} (\conL{\fp})$,
we conclude that $\cng o = m (\conL{\fp})$. 
Since $[o,i]$ is a simple sublattice of $L$,
it follows that $\cng o = i (\conL{\fp})$, 
proving one direction of the equivalence.

Second, assume that $\cng o = i (\conL{\fp})$. Then 
$\cng m = i (\conL{\fp})$. Then there is a prime-projectivity,
a sequence of prime-perspectivities starting at $\fp$
and ending at the prime interval $[m, i]$. 
In this sequence, there is a first prime interval $\fr'$
containing $m$---it is not $\fp$, since $\fp \ci K$.  Thus there is an immediate  previous prime
interval $\fr$ in this sequence, which by duality, we may assume is
prime-perspective down to~$\fr'$.
By the choice of $\fr'$, $0_{\fr}$, $1_{\fr} \in K$ 
and, by Lemma~\ref{simpL},
\begin{equation}\label{eqE}
\cngd 0_{\fr} = 1_{\fr} (\conK{\fp}).
\end{equation}

The element $m$ is doubly-irreducible, 
so it occurs in only two prime intervals of~$L$, 
the prime interval $[m,i]$ and the prime interval $[o,m]$. 
So either $\fr' = [m,i]$ or $\fr' = [o,m]$.

We first consider the case $\fr' = [m,i]$. Then $0_{\fr} \mm i \leq m$. Since $m$ is 
meet-irreducible and $m_* = o$, we get that
$0_{\fr} \mm i \leq o$. 
But then $\fr$ is prime-projective down to $[a, i]$,  where $a$ is
any of the other atoms of the covering {multi-diamond} $[o,i]$---see Figure~\ref{sm1F}. 
So we can take $\fq = [a,i]$.
Observe that, in this case, $\fq$is actually an interval in $[o,i]$.

\begin{figure}[b!]
\centerline{\includegraphics{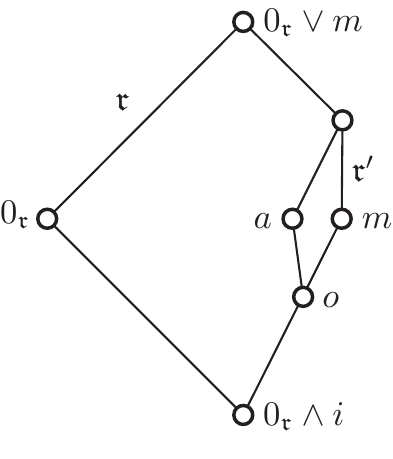}}
\caption{$\fr' = [m,i]$ case of Lemma~\ref{sm0L}}\label{sm1F}
\end{figure}

We are left with the case $\fr' = [o,m]$. 
Since $m \nin \fr$, it follows that $1_{\fr} > m$, 
and so $1_{\fr} \geq i = m^*$. 
Also, $0_{\fr} \mm i < i$, since
$0_{\fr} \mm  m \leq o < m$. 
Let $b$ be a lower cover  of $i$ in the interval $[0_{\fr} \mm i, i]$. 
By \eqref{eqE}, $\cng 0_{\fr} \mm i = i (\conK{\fp})$, 
and so $\cng b = i (\conK{\fp})$---see Figure~\ref{sm2F}. 
We then set $\fq = [b, i]$, concluding the proof. Observe
that in this case, $\fq$ need not be in $[o,i]$ 
since $b$ need not be~$\geq o$.
\end{proof}

\begin{figure}[thb]
\includegraphics{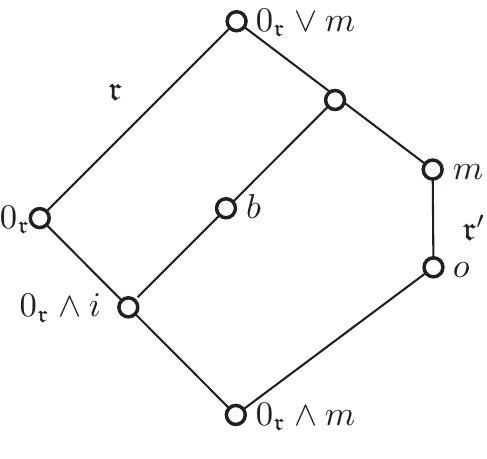}
\caption{$\fr' = [o,m]$ case of Lemma~\ref{sm0L}}\label{sm2F}
\end{figure}

If $L$ is semimodular, then we need not worry about the dual possibility.

\begin{lemma}\label{smL}
Let us further assume that $L$ is semimodular.  
Let $\fp$ be a prime interval in $K$. 
Then $\cng o = i (\conL{\fp})$ if{f} there is a
prime interval $\fq$ in $K$ 
with $\cng 0_{\fq} = 1_{\fq} (\conK{\fp})$ and $1_{\fq} = i$.

If $0_\fq \nin [o,i]$, 
then $\fq$ swings in $L$ to some prime interval $\fr$
in $[o,i] - \set{m}$.
\end{lemma}
\begin{proof}
We first
consider the case when there is a prime interval $\fq'$ in $K$ with
\begin{equation*}
\cngd 0_{\fq'} = 1_{\fq'} (\conK{\fp})
\end{equation*} 
and $0_{\fq'} = o$. Now $1_{\fq'}$ is distinct from $m$. 
By semimodularity, $1_{\fq'} \jj m$ covers $m$ and~$1_{\fq'}$. 
So $1_{\fq'} \jj m = i = m^*$. 
Then
\begin{equation*}
i = \cngd 1_{\fq'} \jj m = 1_{\fq'} (\conK{\fp}).
\end{equation*}
Setting $\fq = [1_{\fq'}, i]$
and noting that  then $0_\fq = 1_{\fq'} \in [o,i]$,
completes the proof in this case.

Otherwise, by Lemma~\ref{sm0L}, there is a prime interval $\fq$ in $K$ with
\begin{equation*}
\cngd 0_{\fq} = 1_{\fq} (\conK{\fp})
\end{equation*}
and with $1_\fq = i$. If $0_\fq \nin [o,i]$, 
then we apply Lemma~\ref{genSL}, with $x = 0_\fq$, to get
the prime interval $\fr = [a,i]$ in $[o,i] - \set{m}$, where $\fq$ swings to $\fr$ by use
of $m$ (and thus the swing is in $L$, but not necessarily in $K$).
\end{proof}

\section{The proof of the General Swing Lemma}\label{S:Swing}

Our proof of the General Swing Lemma for a planar semimodular lattice
will be by induction on the number of tabs in the lattice. 
We then need the following very easy result.

\begin{lemma}\label{veL}
Let $L$ be a finite semimodular lattice containing a doubly-irredu\-cible element $m$. Then
the sublattice $K = L - \set{m}$ is also semimodular.
\end{lemma}
\begin{proof}
Let $a, b, c \in K$ with $b\neq c$ covering $a$.
We have to show that $b$ and $c$ cover $a$ in~$L$; see G.~Gr\"atzer~\cite[Theorem 375]{LTF}.

Assume to the contrary that $b$ does not cover $a$ in $L$. 
Then $a \prec m \prec b$ in $L$. 
So~$m$ and $c$ are distinct upper covers of $a$. 
But $L$ is semimodular,
so $m \jj c$ covers both $m$ and $c$---see Figure~\ref{veF}. 
Since it is meet-irreducible, $m$ has only one upper cover, 
and so $m \jj c = b$. 
Thus, $b$ covers $c$, that is, $a \prec c \prec b$ in~$K$, 
contradicting $a \prec b$ in~$K$.
\begin{figure}[htb]
\centerline{\includegraphics{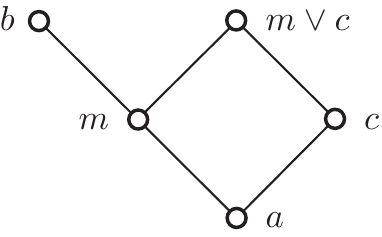}}
\caption{Lemma~\ref{veL}: $a \prec m \prec b$}\label{veF}
\end{figure}

Thus, $b$ and $c$ still cover $a$ in $L$. Then, by the semimodularity of $L$, $b \jj c$ covers
$b$ and $c$  in $L$ and so, certainly, in $K$. Consequently, $K$ is semimodular.
\end{proof}
We now prove the General Swing Lemma.

\begin{proof}[Proof of Lemma~\ref{L:GSL}]
First, let $\fp$, $\fq$ be prime intervals in $L$, and let there be such sequences
$\fp_0 = \fp, \fp_1, \dots, \fp_n$ and $\fq_0, \fq_1, \dots, \fq_n = \fq$.
By Lemma~\ref{collapse2L},  each $\fq_k$ is collapsed by $\con{\fp_k}$, and, by Lemma~\ref{m3L},
each $\fp_{k+1}$ is collapsed by $\con{\fq_k}$. 
Then $\fq$ is collapsed by $\con{\fp}$.

We now proceed in the other direction. Let $\fp$ and $\fq$ be
distinct prime intervals in $L$,
and let $\fq$ be collapsed by $\con{\fp}$. 
We proceed by induction on
the number of tabs of $L$.

Note that in a covering {multi-diamond} in any planar lattice $L'$ each atom, except possibly the
leftmost and rightmost, is  doubly-irreducible in $L'$, that is, 
it is a tab of $L'$.
So if there are no tabs in the lattice $L$, 
then $L$ has no covering {multi-diamond}, that is, 
is slim. So by the Swing Lemma,
$\fp$ is SPS projective to $\fq$. 
We then get the conclusion if there are no tabs.

Now let there be at least one tab in $L$. Take such a tab $m$,
an atom of a covering {multi-diamond} $[o,i]$ in $L$, and set $K = L - \set{m}$. Then,
by Lemma~\ref{veL}, $K$ is a planar semimodular lattice, and,
by the induction hypothesis, the General Swing Lemma holds for the lattice $K$.

If $m \in \fp$, then let $\fp'$ be a prime interval in $[o,i]$ that does not contain $m$.
Then by definition, $\fp$ switches to $\fp'$, and, since $\conL{\fp} = \conL{\fp'}$,
the prime interval $\fq$ is collapsed by $\conL{\fp'}$. 
Similarly, if $m \in \fq$, then there is a
prime interval $\fq' \ci [o,i]$ that switches to $\fq$, that does not contain $m$, but is collapsed by~$\fp$.

Thus to prove the General Swing Lemma, it suffices to assume that $m$ is an element of neither
$\fp$ nor $\fq$, that is, that $\fp$ and $\fq$ lie in $K$. If $\fq$ is collapsed by $\conK{\fp}$, then we are done, 
since $K$ satisfies the General Swing Lemma.

Otherwise, since $\fq$ is collapsed by $\conL{\fp}$,
it follows that $\conL{\fp}\restr K \neq \conK{\fp}$.
Since $\conL{\fp} = \conL{\conK{\fp}}$, we conclude by Lemma~\ref{mainL} that
\begin{equation}\label{firstE}
   \cngd o = i (\conL{\fp})
\end{equation}
and that
\begin{equation}\label{midE}
\conL{\fp}\restr K = \conK{\fp} \jj_K \conK{o,i}.
\end{equation}
By \eqref{firstE} and Lemma~\ref{smL}, there is a prime interval $\fr$ in $K$ with $1_\fr = i$
that is collapsed in $K$ by $\conK{\fp}$. Since
the General Swing Lemma holds for $K$, there are sequences of prime intervals in $K$, 
$\fp_0 = \fp, \fp_1, \dots, \fp_n$ and
$\fq_0, \fq_1, \dots, \fq_n = \fr$,
where $\fp_k$ is SPS projective to $\fq_k$ in $K$
for $k = 0,\dots, n$, and $\fq_k$ switches to $\fp_{k+1}$ 
in $K$ for $k=0,\dots, n-1$, provided that $n > 0$.

If $\fr \nci [o,i]$, then by  Lemma~\ref{smL}, $\fr$ swings in $L$ to some prime interval
$\fr'$ in~$[o,i] \ii K$. Then since $\fp_n$ is SPS projective in $K$ to $\fr$, it follows that $\fp_n$
is SPS projective in $L$ to $\fr'$, a prime interval in $[o, i] \ii K$.  So if $\fr$ is not in $[o,i]$,
we can replace $\fr$ by $\fr'$.

In either event, we have sequences of prime intervals in $K$, 
$\fp_0 = \fp, \fp_1, \dots, \fp_n$ and
$\fq_0, \fq_1, \dots, \fq_n$
where $\fp_k$ is SPS projective to $\fq_k$ in $L$
for $k = 0,\dots, n$, and $\fq_k$ switches to $\fp_{k+1}$ in
$K$ for
$k=0,\dots, n-1$, provided that $n > 0$, and where $\fq_n$ is an interval in $[o,i] \ii K$.

Since $\fq$ lies in $K$ and is collapsed
 by $\conL{\fp}$, it follows from \eqref{midE} that $\fq$ is collapsed by
$\conK{\fp} \jj_K \conK{o,i}$.
Choose an atom $a$ of the multi-diamond $[o,i]$ distinct from $m$ and not an element of
$\fq_n$.
Then
\begin{equation*}
\conK{\fq} \leq  \conK{\fp} \jj_K \conK{o,i} = \conK{\fp} \jj_K \conK{o,a} \jj_K \conK{a,i}.
\end{equation*}
(This is the best we can hope for, since $[o,i] \ii  K$ need not be simple.)
Since $\conK{\fq}$ is join-irreducible
and is not $\leq \conK{\fp}$, we conclude that
\begin{equation}\label{secondE}
\conK{\fq} \leq \conK{\fs} \text{ where } \fs = [0,a] \text{ or } \fs = [a,i].
\end{equation}
Since $a \nin \fq_n$, it follows that $\fq_n$  and $\fs$ are distinct prime intervals in $[o,i]$.
 Thus, $\fq_n$ switches in $L$ to $\fs$.

We now proceed to extend the sequences $\fp_0, \dots, \fp_n$
and $\fq_0, \dots, \fq_n$.

Set $\fp_{n+1} = \fs$. Then $\fq_n$ switches in $L$ to $\fp_{n+1}$.
By \eqref{secondE}, since $K$ satisfies the General Swing Lemma, there are
sequences of prime intervals
in $K$, $\fp_{n+1} = \fs,  \dots, \fp_r$ and
$\fq_{n+1}, \dots, \fq_r = \fq$ where $\fp_k$ is SPS projective in $K$
to $\fq_k$ for $k = n+1, \dots, r$ and $\fq_k$ switches in $K$ to $\fp_{k+1}$ for
$k=n+1, \dots, r-1$. Combining these sequences, we get the desired sequences
$\fp_0 = \fp, \dots, \fp_{n+1}, \dots \fp_r$ and
$\fq_0, \dots, \fq_n, \dots, \fq_r = \fq$, establishing the General Swing Lemma for $L$,
and so concluding the proof of the induction step.

Thus the proof of the General Swing Lemma for planar semimodular lattices is concluded.
\end{proof}

One should note that, in the General Swing Lemma, some of the SPS projectivities could be trivial---there could
be two successive switches in different multi-diamonds, or we could start off with a switch rather
than a proper SPS-projectivity.

\section{Discussion}\label{S:Discussion}
In the book \cite{CFL2}, Section 24.4 discusses 
some consequences of the Swing Lemma. We take up here only one aspect of this topic.
\begin{lemma}\label{L:disc}
Let $L$ be a planar semimodular lattice, and let $\bga,\bgb\in\Conj L$ be join-irreducible congruences. If $\bgb<\bga$, then there exist prime intervals $\fp$ and $\fq$ in $L$ such that 
\begin{enumeratei}
\item $\fp$ swings to $\fq$,
\item $\bgb\leq \con{\fq} <\con{\fp}\leq \bga$, and
\item $\con{\fp}$ covers $\con{\fq}$ in the order $(\Conj L;\leq)$.
\end{enumeratei}
\end{lemma}

\begin{figure}[htb]
\centerline{\includegraphics{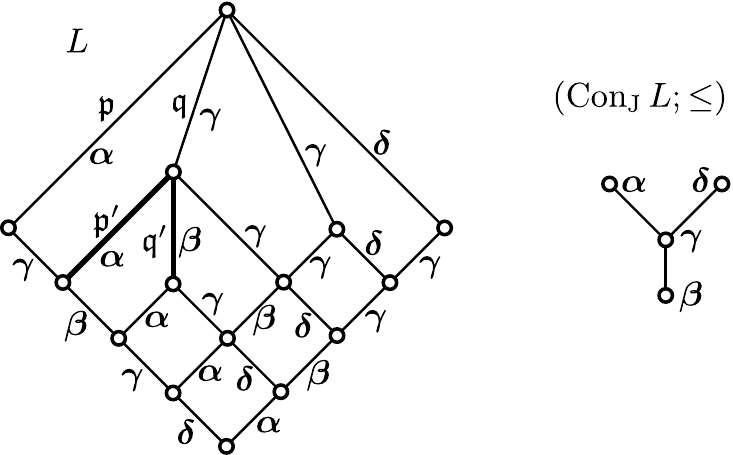}}
\caption{A swing need not give covering}\label{czggghlexmpl}
\end{figure}

Figure~\ref{czggghlexmpl} shows that this lemma does not follow as immediately from the General Swing Lemma. This figure defines a planar semimodular lattice $L$, which is actually an SPS lattice. Although the Swing Lemma in itself is appropriate to determine the order $(\Conj L;\leq)$ as given in the figure, the methods of G.~Cz\'edli~
\cite{gC14} and, mainly, \cite{czganote} are more effective. 
The prime intervals of $L$ are labeled by the principal congruences they generate. 
If we represent $\bga$ and $\bgb$ by the thick edges as $\bga=\con{\fp'}$ and $\bgb=\con{\fq'}$, then the General Swing Lemma applied for $\fp'$ and $\fq'$ yields that
$n=0$, $\fp_0=\fp'$, $\fq_0=\fq'$, and the SPS projectivity described in \eqref{Eq:SPS} consist of a single swing of $\fp'$ to $\fq'$. Hence, we cannot obtain a required covering in $\Conj L$ in this way. Appropriate $\fp$ and $\fq$ for Lemma~\ref{L:disc} are given in the figure, and the proof runs as follows.

\begin{proof}[Proof of Lemma~\ref{L:disc}]
With  $\bga,\bgb\in\Conj L$ as in the lemma, pick $\bga'$ and $\bgb'$ in $\Conj L$ such that $\bga'$ covers $\bgb'$ in $(\Con L;\leq)$,  $\bgb\leq \bgb'$, and $\bga'\leq \bga$. Since a join-irreducible congruence of a finite lattice is always generated by a prime interval, we can pick prime intervals $\fp'$ and $\fq'$ in $L$ such that 
$\bga'=\con{\fp'}$ and $\bgb'=\con{\fq'}$. Since $\con{\fq'}=\bgb'< \bga'=\con{\fp'}$, the prime interval $\fq'$ is collapsed by $\con{\fp'}$. Hence, the General Swing Lemma yields a sequence \eqref{Eq:SPS} from $\fp'$ to $\fq'$ such that, for each $j\in\set{0,\dots,n-1}$, $\fr_j$ is perspective to or swings to or switches to $\fr_{j+1}$.
Of course, this sequence  satisfies that
\begin{equation*}
   \con{\fp'} = \con{\fr_0} \geq \con{\fr_1}  
      \geq \cdots \geq \con{\fr_{n-1}} \geq \con{\fr_n} = \con{\fq'}.
\end{equation*}
If for an integer $j$ with $0 \leq j < n$,
the prime interval $\fr_j$ is perspective to or switches to the prime interval $\fr_{j+1}$, then $\con{\fr_j} = \con{\fr_{j+1}}$. Therefore, since $\con{\fp'}$ covers $\con{\fq'}$ in $(\Conj L;\leq)$ and all the $\con{\fr_j}$ belong to $\Conj L$, there is a unique $i\in\set{1,\dots,n-1}$ such that $\con{\fp'}=\con{\fr_i}$, $\fr_i$ swings to $\fr_{i+1}$, and $\con{\fr_{i+1}}=\con{\fq'}$. Therefore,  we can let $\fp:=\fr_i$ and $\fq:=\fr_{i+1}$.
\end{proof}

We conclude this paper with one more variant of the General Swing Lemma.
This variant makes no direct reference to planarity.

Let $L$ be a finite lattice and let $m$ be a tab of $L$.
Then $K = L - \set{m}$ is a sublattice of $L$. 
We obtain $K$ by \emph{stripping} the tab $m$ from $L$. 
A finite lattice $L$ is \emph{stripped}, if it has no tabs.
It is clear that by consecutive stripping of tabs, 
we obtain from $L$ a stripped (sub) lattice $L^\textup{s}$. 
If we obtain the stripped sublattices $K_1$ and $K_2$ from
the finite lattice $L$, then $K_1$ and $K_2$ are isomorphic; 
this follows from the proof of Lemma 4.1 
in G.~Cz\'edli and E.\,T.~Schmidt~\cite{CS13}.
For instance, starting from $\SM 3$, 
we obtain three different stripped sublattices, 
all of them isomorphic to $\SC{2}^2$.

\begin{lemma}[Reduction Lemma]\label{reductlemma}
Let $L$ be a finite semimodular lattice and 
let $K$ be a stripped sublattice of $L$.
If $K$ satisfies the Swing Lemma, 
then $L$ satisfies the General Swing Lemma.
\end{lemma}

A proof of the Reduction Lemma is implicit in Section~\ref{S:Swing}.
Observing that in the Reduction Lemma, 
if $L$ is planar and semimodular, 
then $K$ is slim, planar, and semimodular,
we conclude that the Reduction Lemma implies 
the General Swing Lemma.

Finally, we know that gluing preserves semimodularity; see, for example, E.~Fried, G.~Gr\"atzer, and E.\,T.~Schmidt~\cite[Theorem 27]{FrGrSch93} for a stronger statement. Hence, if we glue a non-planar distributive lattice and a planar but non-slim semimodular lattice, then we obtain a nontrivial example in the  scope of the Reduction Lemma but not of the General Swing Lemma.

\end{document}